\documentclass{amsart}
\usepackage[centertags]{amsmath}
\usepackage{amsfonts}
\usepackage{amssymb}
\usepackage{amsthm}
\usepackage[top=3cm, bottom=3cm,left=3cm, right=3cm]{geometry}
 \newtheorem{thm}{Theorem}[section]
 \newtheorem{cor}[thm]{Corollary}
 \newtheorem{lem}[thm]{Lemma}
 
 \theoremstyle{definition}
 \newtheorem{defn}[thm]{Definition}
 \theoremstyle{remark}
 \newtheorem{rem}[thm]{Remark}
 \numberwithin{equation}{section}
\begin{document}
\title[Classification of quasi-trigonometric solutions]{Classification of quasi-trigonometric solutions of the classical
Yang--Baxter equation}
\author{Iulia Pop and Alexander Stolin} 

\address{Department of Mathematical Sciences, University of Gothenburg, Sweden}
\email{iulia@math.chalmers.se, astolin@math.chalmers.se}

\begin{abstract}
It was proved by Montaner and Zelmanov that 
up to classical twisting Lie bialgebra structures 
on $\mathfrak{g}[u]$ fall into four classes.
Here $\mathfrak{g}$ is a simple complex finite-dimensional 
Lie algebra.
It turns out that classical twists within one of these
four classes are in a one-to-one correspondence with the 
so-called quasi-trigonometric solutions of the
classical Yang-Baxter equation.
In this paper we give a complete list of the 
quasi-trigonometric solutions in terms of 
sub-diagrams of the certain Dynkin diagrams
related to $\mathfrak{g}$.
We also explain how to quantize the corresponding
Lie bialgebra structures.
\end{abstract}
\maketitle
\section{Introduction}
The present paper constitutes a step towards the classification 
of quantum groups. We describe an algorithm for the quantization 
of all Lie bialgebra structures on the polynomial Lie algebra 
$P=\mathfrak{g}[u]$, where $\mathfrak{g}$ is a simple complex
finite-dimensional Lie algebra.  

Lie bialgebra structures 
on $P$, up to so-called \textit{classical twisting}, have been classified by 
F. Montaner and E. Zelmanov in \cite{MZ}. We recall that given a Lie 
co-bracket $\delta$ on $P$, a \textit{classical twist} 
is an element $s\in P\wedge P$ such that 
\begin{equation}
\mathrm{CYB}(s)+\mathrm{Alt}(\delta\otimes \mathrm{id})(s)=0,
\end{equation}
where $\mathrm{CYB}$ is the l.h.s. of the classical Yang-Baxter equation.

We also note that a classical twist does not change the classical double 
$D_{\delta}(P)$
associated to a given Lie bialgebra structure $\delta$. 
If $\delta^s$ is the twisting co-bracket
via $s$, then the Lie bialgebras $(P,\delta)$ and $(P,\delta^s)$ are in the same class, i.e. there exists a Lie algebra isomorphism between 
$D_{\delta}(P)$ and $D_{\delta^s}(P)$, preserving the canonical forms and compatible with the canonical embeddings of $P$ into the doubles.

According to the results of Montaner and Zelmanov, there are four Lie bialgebra
structures on $P$ up to classical twisting. Let us present them:
\vspace{0.5cm}

\textbf{Case 1.} Consider $\delta_1=0$. Consequently, $D_1(P)=P+\varepsilon P^{*}$, where $\varepsilon^2=0$. The symmetric nondegenerate invariant form $Q$ is 
given by the canonical pairing between $P$ and $\varepsilon P^{*}$.

Lie bialgebra structures which fall in this class are the elements $s\in P\wedge P$ satisfying $\mathrm{CYB}(s)=0$. Such elements are in a one-to-one correspondence 
with finite-dimensional quasi-Frobenius Lie subalgebras of $P$.
\vspace{0.5cm}

\textbf{Case 2.} Let us consider the co-bracket $\delta_2$ given by 
\begin{equation}
\delta_2(p(u))=[r_2(u,v),p(u)\otimes 1+1\otimes p(v)],
\end{equation}
where $r_2(u,v)=\Omega/(u-v)$. Here $\Omega$ denotes the quadratic Casimir element on $\mathfrak{g}$. 

It was proved in \cite{S1} that the associated classical double is $D_2(P)=\mathfrak{g}((u^{-1}))$, together with the canonical invariant form 
\begin{equation}\label{eq1}
Q(f(u),g(u))=Res_{u=0}K(f,g),
\end{equation}
where $K$ denotes the Killing form of the Lie algebra $\mathfrak{g}((u^{-1}))$ over $\mathbb{C}((u^{-1}))$. 

Moreover, the Lie bialgebra structures which are obtained by twisting $\delta_2$ are in a one-to-one correspondence 
with so-called \textit{rational solutions} of the CYBE, according to \cite{S1}.
\vspace{0.5cm}

\textbf{Case 3.} In this case, let us consider the Lie bialgebra structure given by 
\begin{equation}\label{Delta_3}
\delta_3(p(u))=[r_3(u,v),p(u)\otimes 1+1\otimes p(v)],
\end{equation}
with $r_3(u,v)=v\Omega/(u-v) +\Sigma_{\alpha} e_{\alpha}\otimes f_{\alpha}+{\frac{1}{2}}\Omega_0$,
where ${e_{\alpha},  f_{\alpha}}$ are root vectors of $\mathfrak{g}$ and $\Omega_0$ is the Cartan part
of  $\Omega$.

It was proved in \cite{KPST} that the associated classical double is 
$D_3(P)=\mathfrak{g}((u^{-1}))\times\mathfrak{g}$, together with the invariant nondegenerate form $Q$ defined by 
\begin{equation}\label{eq2}
Q((f(u),a),(g(u),b))=K(f(u),g(u))_0-K(a,b),
\end{equation}
where the index zero means that one takes the free term in the series expansion. 
According to \cite{KPST}, there is a one-to-one correspondence between Lie bialgebra structures which are obtained by twisting 
$\delta_3$ and so-called \textit{quasi-trigonometric solutions} of the CYBE.

\vspace{0.5cm}

\textbf{Case 4.} We consider the co-bracket on $P$ given by 
\begin{equation}
\delta_4(p(u))=[r_4(u,v),p(u)\otimes 1+1\otimes p(v)],
\end{equation}
with $r_4(u,v)=uv\Omega/(v-u)$.

It was shown in \cite{SY} that the classical double associated to the Lie bialgebra structure $\delta_4$ is $D_4(P)=\mathfrak{g}((u^{-1}))\times(\mathfrak{g}\otimes
\mathbb{C}[\varepsilon])$, where $\varepsilon^2=0$. The form $Q$ is described as follows: if $f(u)=\sum_{-\infty}^{N} a_ku^k$ and
$g(u)=\sum_{-\infty}^{N}b_ku^k$, then 
\[
Q(f(u)+A_0+A_1\varepsilon,g(u)+B_0+B_1\varepsilon)=
Res_{u=0}u^{-2}K(f,g)-K(A_0,B_1)-K(A_1,B_0).
\]

Lie bialgebra structures which are in the same class as $\delta_4$ are in a one-to-one correspondence with \textit{quasi-rational r-matrices}, as it was proved in \cite{SY}. 

Regarding the quantization of these Lie bialgebra structures on $P$, the following conjecture stated in \cite{KPST} and proved by G. Halbout in \cite{H} plays
a crucial role.
\begin{thm} \label{Hal} Any classical twist can be extended to a quantum twist, i.e., 
if $(L,\delta)$ is any Lie bialgebra, $s$ is a classical twist, and $(A,\Delta,\varepsilon)$ is a quantization of $(L,\delta)$, there exists 
$F\in A\otimes A$ such that 

(1) $F=1+O(\hbar)$ and $F-F^{21}=\hbar s+O(\hbar^2)$,

(2) $(\Delta\otimes \mathrm{id})(F)F^{12}-(\mathrm{id}\otimes \Delta)(F)F^{23}=0$,

(3) $(\varepsilon\otimes \mathrm{id})(F)=(\mathrm{id} \otimes \varepsilon)(F)=1$.

Moreover gauge equivalence classes of quantum twists for $A$ are in bijection
with gauge equivalence classes of $\hbar$-dependent classical twists $s_\hbar=\hbar s_1+O(\hbar^2)$ for $L$. 

\end{thm}

Let us suppose that we have a Lie bialgebra structure $\delta$ on $P$. 
Then $\delta$ is obtained by twisting one of the four structures $\delta_i$ from 
Cases 1--4. This above theorem implies that in order to find a quantization for $(P,\delta)$, 
it is sufficient to determine the quantization of $\delta_i$ and then find the quantum twist whose classical 
limit is $s$. Let us note that the quantization of $(P, \delta_3)$ is well-known. 
The corresponding quantum algebra was introduced by V. Tolstoy in \cite{tol}
and it is denoted by $U_q  (\mathfrak{g}[u])$.
 
The quasi-trigonometric solutions of the CYBE were
studied in \cite{KPSST}, where it was proved that they fall into classes, which are in a one-to-one correspondence
with vertices of the extended Dynkin diagram of $\mathfrak{g}$. Let us consider corresponding roots, namely simple roots
$\alpha_1,\alpha_2,\cdots \alpha_r$ and $\alpha_0=-\alpha_{\rm{max}}$.
In \cite{KPSST}
quasi-trigonometric solutions corresponding to the simple roots which have coefficient one in the decomposition of the
maximal root
were classified. It was also proved there that quasi-trigonometric solutions corresponding to
$\alpha_0$ are in a one-to-one correspondence with constant solutions of the modified CYBE
classified in \cite{BD} and the polynomial part of these solutions is constant. 
The aim of our paper is to obtain a 
complete classification of quasi-trigonometric solutions of the CYBE. In particular,
we describe all the quasi-trigonometric solutions with non-trivial polynomial
part for  $\mathfrak{g}=o(5)$.

\section{Lie bialgebra structures associated with quasi-trigonometric solutions}\begin{defn}
A solution $X$ of the CYBE is called \emph{quasi-trigonometric} if
it is of the form $X(u,v)=v\Omega/(u-v)+p(u,v)$, where $p$ is a
polynomial with coefficients in $\mathfrak{g}\otimes\mathfrak{g}$.
\end{defn}
The class of quasi-trigonometric solutions is closed under 
\emph{gauge transformations}. We first need to introduce the following notation: 
Let $R$ be a commutative ring and let $L$ be a Lie algebra over $R$. Let us denote by
$\mathrm{Aut}_R (L)$ the group of automorphisms of $L$ over $R$. In other words
we consider such automorphisms of $L$, which satisfy the condition $f(rl)=rf(l)$, where
$r\in R,\ l\in L.$ 

At this point we note that there exists a natural embedding
$$\mathrm{Aut}_{\mathbb{C}[u]}(\mathfrak{g}[u])\hookrightarrow\mathrm{Aut}_{\mathbb{C}((u^{-1}))}(\mathfrak{g}((u^{-1}))),$$
defined by the formula
$$\sigma(u^{-k}x)=u^{-k}\sigma(x)\,,$$
for any $\sigma\in\mathrm{Aut}_{\mathbb{C}[u]}(\mathfrak{g}[u])$ and $x\in\mathfrak{g}[u]$. 

 Now if $X$ is a
quasi-trigonometric solution and 
$\sigma(u)\in \mathrm{Aut}_{\mathbb{C}[u]}(\mathfrak{g}[u])$, one can check that the function $Y(u,v):=(\sigma(u)\otimes\sigma(v))(X(u,v))$ is
again a quasi-trigonometric solution. $X$ and $Y$ are said to be
\emph{gauge equivalent}.

\begin{thm}
There exists a natural one-to-one correspondence between
quasi--trigonometric solutions of CYBE for $\mathfrak{g}$ and
linear subspaces $W$ of $\mathfrak{g}((u^{-1}))\times\mathfrak{g}$
which satisfy the following properties:

(1) $W$ is a Lie subalgebra of
$\mathfrak{g}((u^{-1}))\times\mathfrak{g}$ and $W\supseteq
u^{-N}\mathfrak{g}[[u^{-1}]]$ for some positive integer $N$.

(2)
$W\oplus\mathfrak{g}[u]=\mathfrak{g}((u^{-1}))\times\mathfrak{g}$.

(3) $W$ is a Lagrangian subspace of
$\mathfrak{g}((u^{-1}))\times\mathfrak{g}$ with respect to the
invariant bilinear form $Q$ given by (\ref{eq2}).
\end{thm}

Let $\sigma(u)\in \mathrm{Aut}_{\mathbb{C}[u]}(\mathfrak{g}[u])$.
Let $\widetilde{\sigma}(u)=\sigma(u)\oplus \sigma(0)$ be the induced automorphism of $\mathfrak{g}((u^{-1}))\times\mathfrak{g}$.

\begin{defn}
We will say that $W_{1}$ and $W_{2}$ are \emph{gauge equivalent} if  there exists
$\sigma(u)\in \mathrm{Aut}_{\mathbb{C}[u]}(\mathfrak{g}[u])$ such that $W_{1}=\widetilde{\sigma}(u)W_{2}$.
\end{defn}
It was checked in \cite{KPST} that two quasi-trigonometric solutions are gauge equivalent if and only if the corresponding subalgebras are gauge equivalent.

Let $\mathfrak{h}$ be a Cartan subalgebra of $\mathfrak{g}$ with
 the corresponding set of roots $R$ and a choice of simple
 roots $\Gamma$. Denote by $\mathfrak{g}_{\alpha}$ the root space
 corresponding to a root $\alpha$. Let $\mathfrak{h}(\mathbb{R})$
 be the set of all $h\in\mathfrak{h}$ such that
 $\alpha(h)\in\mathbb{R}$ for all $\alpha\in R$.
 Consider the valuation on $\mathbb{C}((u^{-1}))$ defined by
 $v(\sum_{k\geq n}a_{k}u^{-k})=n$. For any root $\alpha$ and any
 $h\in\mathfrak{h}(\mathbb{R})$, set
 $M_{\alpha}(h)$:=$\{f\in\mathbb{C}((u^{-1})):v(f)\geq \alpha(h)\}$.
 Consider
\begin{equation}
 \mathbb{O}_{h}:=\mathfrak{h}[[u^{-1}]]\oplus(\oplus_{\alpha\in R}M_{\alpha}(h)\otimes\mathfrak{g}_{\alpha}).
\end{equation}

As it was shown in \cite{KPSST}, any maximal order $W$ which corresponds to a quasi-trigonometric solution of the CYBE, can be embedded (up to some gauge equivalence) into $\mathbb{O}_{h}\times\mathfrak{g}$. Moreover $h$ may be taken as a vertex of the standard simplex 
$\Delta_{st}= \{h\in\mathfrak{h}(\mathbb{R}):$ $\alpha(h)\geq 0$ for
all $\alpha\in\Gamma$ and $\alpha_{\max}\leq 1\}$.

Vertices of the above simplex correspond to vertices of the extended Dynkin diagram of
$\mathfrak{g}$, the correspondence being given by the following rule:
\[0\leftrightarrow\alpha_{\max}\]\[h_{i}\leftrightarrow\alpha_{i},\]
where $\alpha_{i}(h_{j})=\delta_{ij}/k_{j}$ and $k_{j}$ are given by the relation $\sum
k_{j}\alpha_{j}=\alpha_{\max}$. We will write $\mathbb{O}_{\alpha}$ instead of
$\mathbb{O}_{h}$ if $\alpha$ is the root which corresponds to the vertex $h$.

By straightforward computations, one can check the following two results:

\begin{lem}\label{Lemma1}
Let $R$ be the set of all roots and $\alpha$ an arbitrary simple root. Let $k$ be the coefficient of $\alpha$ in the decomposition of $\alpha_{\max}$. 

For each $r$, $-k\leq r\leq k$, let $R_{r}$ denote the set of all roots which contain $\alpha$ with coefficient $r$. Let $\mathfrak{g}_{0}=\mathfrak{h}\oplus\sum_{\beta\in R_{0}}\mathfrak{g}_\beta$ and $\mathfrak{g}_r=\sum_{\beta\in R_r}\mathfrak{g}_\beta$. Then
\begin{equation}
\mathbb{O}_\alpha=\sum_{r=1}^k u^{-1}\mathbb{O}\mathfrak{g}_r+
\sum_{r=1-k}^0 \mathbb{O}\mathfrak{g}_r+u\mathbb{O}\mathfrak{g}_{-k},
\end{equation}
where $\mathbb{O}:=\mathbb{C}[[u^{-1}]]$. 
\end{lem}

\begin{lem} Let $\alpha$ be a simple root and $k$ its coefficient in the decomposition of $\alpha_{\max}$. Let $\Delta_\alpha$ denote the set of all pairs $(a,b)$, $a\in\mathfrak{g}_0+\mathfrak{g}_{-k}$, $b\in\mathfrak{g}_0+\mathfrak{g}_{-1}+...+\mathfrak{g}_{-k}$, $a=a_0+a_{-k}$, $b=b_0+b_{-1}+...+b_{-k}$ and $ a_0=b_0$. Then

(i) The orthogonal complement of $\mathbb{O}_\alpha\times\mathfrak{g}$ with respect to $Q$ is given by 
\begin{equation}
(\mathbb{O}_\alpha\times\mathfrak{g})^\perp=
\sum_{r=-k}^{-1}\mathbb{O}\mathfrak{g}_r+
\sum_{r=0}^{k-1}u^{-1}\mathbb{O}\mathfrak{g}_r+
u^{-2}\mathbb{O}\mathfrak{g}_k.
\end{equation}

(ii) There exists an isomorphism $\sigma$ 
\begin{equation}
\frac{\mathbb{O}_\alpha\times\mathfrak{g}}{(\mathbb{O}_\alpha\times\mathfrak{g})^\perp}\cong(\mathfrak{g}_k\oplus\mathfrak{g}_0\oplus\mathfrak{g}_{-k})\times\mathfrak{g}
\end{equation}
given by 
\[\sigma((f,a)+(\mathbb{O}_\alpha\times\mathfrak{g})^\perp)=(a_0+b_0+c_0,a),\]
where the element $f\in\mathbb{O}_\alpha$ is decomposed according to Lemma \ref{Lemma1}:
\[f=u^{-1}(a_0+a_1u^{-1}+...)+(b_0+b_1u^{-1}+...)+u(c_0+c_1u^{-1}+...)+...,\] 
$a_i\in\mathfrak{g}_k$, $b_i\in\mathfrak{g}_0$, $c_i\in\mathfrak{g}_{-k}$
and $a\in\mathfrak{g}$. 

(iii) $(\mathbb{O}_{\alpha}\times\mathfrak{g})\cap\mathfrak{g}[u]$ is sent via the isomorphism $\sigma$ to $\Delta_\alpha$. 
\end{lem}

Let us make an important remark. The Lie subalgebra $\mathfrak{g}_k+\mathfrak{g}_0+\mathfrak{g}_{-k}$ of $\mathfrak{g}$ 
coincides with the semisimple Lie algebra whose Dynkin diagram is obtained from the extended Dynkin diagram of $\mathfrak{g}$ by 
crossing out $\alpha$. Let us denote this subalgebra by $L_\alpha$.
The Lie algebra $L_\alpha\times\mathfrak{g}$ is endowed with the following invariant bilinear form:
\begin{equation}
Q'((a,b),(c,d))=K(a,c)-K(b,d),
\end{equation}
for any $a,c\in L_\alpha$ and $b,d\in\mathfrak{g}$. 

On the other hand, $\mathfrak{g}_0+\mathfrak{g}_{-k}$ is the parabolic subalgebra $P_{-\alpha_{\max}}^{+}$ of $L_\alpha$ which corresponds to $-\alpha_{\max}$.
The Lie subalgebra $\mathfrak{g}_0+\mathfrak{g}_{-1}+...+\mathfrak{g}_{-k}$ is the parabolic subalgebra $P_{\alpha}^{-}$ of $\mathfrak{g}$ which corresponds to the root $\alpha$ and contains the negative Borel subalgebra. Let us also note that $\mathfrak{g}_0$ is precisely the reductive part of $P_{\alpha}^{-}$ and of $P_{-\alpha_{\max}}^{+}$. We can conclude that the set $\Delta_{\alpha}$ consists of all pairs $(a,b)\in P_{-\alpha_{\max}}^{+}\times P_{\alpha}^{-}$ whose reductive parts are equal.

\begin{thm}

Let $\alpha$ be a simple root. There is a one-to-one correspondence between Lagrangian subalgebras $W$ of $\mathfrak{g}((u^{-1}))\times\mathfrak{g}$ which are contained in $\mathbb{O}_\alpha\times\mathfrak{g}$ and transversal to $\mathfrak{g}[u]$, and Lagrangian subalgebras $\mathfrak{l}$ of $L_{\alpha}\times\mathfrak{g}$ transversal to $\Delta_\alpha$ (with respect to the bilinear form $Q'$). 
\end{thm}
\begin{proof}
Since $W$ is a subspace of $\mathbb{O}_{\alpha}\times\mathfrak{g}$, let $\mathfrak{l}$ be its image in $L_{\alpha}\times\mathfrak{g}$. Because $W$ is transversal to $\mathfrak{g}[u]$, one can check that $\mathfrak{l}$ is transversal to the image of $(\mathbb{O}_{\alpha}\times\mathfrak{g})\cap\mathfrak{g}[u]$ in $L_{\alpha}\times\mathfrak{g}$, which is exactly $\Delta_{\alpha}$. The fact that $W$ is Lagrangian implies that $\mathfrak{l}$ is also Lagrangian.

Conversely, if $\mathfrak{l}$ is a Lagrangian subalgebra of $L_{\alpha}\times\mathfrak{g}$ transversal to $\Delta_{\alpha}$, then its preimage $W$ in $\mathbb{O}_\alpha\times\mathfrak{g}$ is transversal to $\mathfrak{g}[u]$ and Lagrangian as well.  

\end{proof}
The Lagrangian subalgebras $\mathfrak{l}$ of $L_{\alpha}\times\mathfrak{g}$ which are transversal to $\Delta_{\alpha}$, can be determined using results of P. Delorme \cite{Del} on the classification of Manin triples. We are interested in determining Manin triples of the form $(Q',\Delta_{\alpha},\mathfrak{l})$. 

 Let us recall Delorme's construction of so-called \textit{generalized Belavin-Drinfeld data}. Let $\mathfrak{r}$ be a finite-dimensional complex, reductive, Lie algebra and $B$ a
symmetric, invariant, nondegenerate bilinear form on $\mathfrak{r}$. The goal in
\cite{Del} is to classify all Manin triples of $\mathfrak{r}$ up to conjugacy under the
action on $\mathfrak{r}$ of the simply connected Lie group $\mathcal{R}$ whose Lie
algebra is $\mathfrak{r}$.

One denotes by $\mathfrak{r}_{+}$ and $\mathfrak{r}_{-}$ respectively the sum of the
simple ideals of $\mathfrak{r}$ for which the restriction of $B$ is equal to a positive
(negative) multiple of the Killing form. Then the derived ideal of $\mathfrak{r}$ is the
sum of $\mathfrak{r}_{+}$ and $\mathfrak{r}_{-}$.

Let $\mathfrak{j}_{0}$ be a Cartan subalgebra of $\mathfrak{r}$, $\mathfrak{b}_{0}$ a
Borel subalgebra containing $\mathfrak{j}_{0}$ and $\mathfrak{b}_{0}'$ be its opposite.
Choose $\mathfrak{b}_{0}\cap\mathfrak{r}_{+}$ as Borel subalgebra of $\mathfrak{r}_{+}$
and $\mathfrak{b}_{0}'\cap\mathfrak{r}_{-}$ as Borel subalgebra of $\mathfrak{r}_{-}$.
Denote by $\Sigma_{+}$ (resp., $\Sigma_{-}$) the set of simple roots of
$\mathfrak{r}_{+}$ (resp., $\mathfrak{r}_{-}$) with respect to the above Borel
subalgebras. Let $\Sigma=\Sigma_{+}\cup\Sigma_{-}$ and denote by $\mathcal{W}=
(H_{\alpha},X_{\alpha},Y_{\alpha})_{\alpha\in\Sigma_+}$ a Weyl system of generators of
$[\mathfrak{r},\mathfrak{r}]$. 
\begin{defn}[Delorme, \cite{Del}]\label{BD data}
One calls $(A,A',\mathfrak{i}_{\mathfrak{a}},\mathfrak{i}_{\mathfrak{a}'})$
\emph{generalized Belavin-Drinfeld data} with respect to $B$ when the following five
conditions are satisfied:

(1) $A$ is a bijection from a subset $\Gamma_{+}$ of $\Sigma_{+}$ on a subset
$\Gamma_{-}$ of $\Sigma_{-}$ such that
\[
B(H_{A\alpha},H_{A\beta})=-B(H_{\alpha},H_{\beta}), \alpha, \beta\in\Gamma_{+}.
\]

(2) $A'$ is a bijection from a subset $\Gamma'_{+}$ of $\Sigma_{+}$ on a subset
$\Gamma'_{-}$ of $\Sigma_{-}$ such that
\[
B(H_{A'\alpha},H_{A'\beta})=-B(H_{\alpha},H_{\beta}),\alpha,\beta\in\Gamma'_{+}.
\]

(3) If $C=A^{-1}A'$ is the map defined on $\rm{dom}(C)=\{\alpha\in\Gamma'_{+}:
A'\alpha\in\Gamma_{-}\}$ by $C\alpha=A^{-1}A'\alpha$, then $C$ satisfies:

For all $\alpha\in \rm{dom}(C)$, there exists a positive integer $n$ such that $\alpha$,...,
$C^{n-1}\alpha\in \rm{dom}(C)$ and $C^{n}\alpha\notin \rm{dom}(C)$.

(4) $\mathfrak{i}_{\mathfrak{a}}$ (resp., $\mathfrak{i}_{\mathfrak{a}'}$) is a complex
vector subspace of $\mathfrak{j}_{0}$, included and Lagrangian in the orthogonal
$\mathfrak{a}$ (resp., $\mathfrak{a}'$) to the subspace generated by $H_{\alpha}$,
$\alpha\in\Gamma_{+}\cup\Gamma_{-}$ (resp., $\Gamma'_{+}\cup\Gamma'_{-}$).

(5) If $\mathfrak{f}$ is the subspace of $\mathfrak{j}_{0}$ generated by the family
$H_{\alpha}+H_{A\alpha}$, $\alpha\in\Gamma_{+}$, and $\mathfrak{f}'$ is defined
similarly, then
\[
(\mathfrak{f}\oplus\mathfrak{i}_{\mathfrak{a}})\cap(\mathfrak{f}'\oplus
\mathfrak{i}_{\mathfrak{a}'})={0}.
\]
\end{defn}

Let $R_{+}$ be the set of roots of $\mathfrak{j}_{0}$ in $\mathfrak{r}$ which are linear
combinations of elements of $\Gamma_{+}$. One defines similarly $R_{-}$, $R'_{+}$ and
$R'_{-}$. The bijections $A$ and $A'$ can then be extended by linearity to bijections
from $R_{+}$ to $R_{-}$ (resp., $R'_{+}$ to $R'_{-}$). If $A$ satisfies condition (1),
then there exists a unique isomorphism $\tau$ between the subalgebra $\mathfrak{m}_{+}$
of $\mathfrak{r}$ spanned by $X_{\alpha}$, $H_{\alpha}$ and $Y_{\alpha}$,
$\alpha\in\Gamma_{+}$, and the subalgebra $\mathfrak{m}_{-}$ spanned by $X_{\alpha}$,
$H_{\alpha}$ and $Y_{\alpha}$, $\alpha\in\Gamma_{-}$, such that
$\tau(H_{\alpha})=H_{A\alpha}$, $\tau(X_{\alpha})=X_{A\alpha}$,
$\tau(Y_{\alpha})=Y_{A\alpha}$ for all $\alpha\in\Gamma_{+}$. If $A'$ satisfies (2), then
one defines similarly an isomorphism $\tau'$ between $\mathfrak{m'}_{+}$ and
$\mathfrak{m'}_{-}$.

\begin{thm}[Delorme, \cite{Del}]\label{Manin}
(i) Let
$\mathcal{B}\mathcal{D}=(A,A',\mathfrak{i}_{\mathfrak{a}},\mathfrak{i}_{\mathfrak{a}'})$
be generalized Belavin-Drinfeld data, with respect to $B$. Let $\mathfrak{n}$ be the sum
of the root spaces relative to roots $\alpha$ of $\mathfrak{j}_{0}$ in
$\mathfrak{b}_{0}$, which are not in $R_{+}\cup R_{-}$. Let $\mathfrak{i}:=
\mathfrak{k}\oplus\mathfrak{i}_{\mathfrak{a}}\oplus\mathfrak{n}$, where
$\mathfrak{k}:=\{X+\tau(X):X\in\mathfrak{m}_{+}\}$.

Let $\mathfrak{n'}$ be the sum of the root spaces relative to roots $\alpha$ of
$\mathfrak{j}_{0}$ in $\mathfrak{b}_{0}'$, which are not in $R'_{+}\cup R'_{-}$. Let
$\mathfrak{i'}:=\mathfrak{k'}\oplus \mathfrak{i}_{\mathfrak{a'}}\oplus\mathfrak{n'}$,
where $\mathfrak{k'}:=\{X+\tau'(X):X\in\mathfrak{m'}_{+}\}$.

Then $(B,\mathfrak{i},\mathfrak{i}')$ is a Manin triple.

(ii) Every Manin triple is conjugate by an element of $\mathcal{R}$ to a unique Manin
triple of this type.
\end{thm}
Let us consider the particular case 
$\mathfrak{r}=L_{\alpha}\times\mathfrak{g}$.
We set $\Sigma_{+}:=(\Gamma^{ext} \setminus \{\alpha\})\times\{0\}$,
$\Sigma_{-}:=\{0\}\times\Gamma$ and $\Sigma:=\Sigma_{+}\cup\Sigma_{-}$. 

Denote by $(X_{\gamma},Y_{\gamma}, H_{\gamma})_{\gamma\in\Gamma}$ a Weyl system of generators for $\mathfrak{g}$ with respect to the root system $\Gamma$. 
Denote $-\alpha_{\max}$ by $\alpha_0$. 
Let $H_{\alpha_0}$ be the coroot of $\alpha_0$. We choose $X_{\alpha_0}\in\mathfrak{g}_{\alpha_0}$, $Y_{\alpha_0}\in\mathfrak{g}_{-\alpha_0}$ such that $[X_{\alpha_0},Y_{\alpha_0}]=H_{\alpha_0}$. 

A Weyl system of generators in $L_{\alpha}\times\mathfrak{g}$ (with respect to the root system $\Sigma$) is the following:
$X_{(\beta,0)}=(X_\beta,0)$, $H_{(\beta,0)}=(H_\beta,0)$, $Y_{(\beta,0)}=(Y_\beta,0)$, for any $\beta\in\Gamma^{ext} \setminus \{\alpha\}$, and 
$X_{(0,\gamma)}=(0,X_{\gamma})$, $H_{(0,\gamma)}=(0,H_\gamma)$, $Y_{(0,\gamma)}=(0,Y_\gamma)$, for any $\gamma\in\Gamma$. 

By applying the general result of Delorme, one can deduce the description of the Manin triples of the form $(Q',\Delta_{\alpha},\mathfrak{l})$. 

First, let us denote by $i$ the embedding
\[\Gamma\setminus \{\alpha\}\hookrightarrow\Gamma^{ext}\setminus\{\alpha\},\]
Recall that $\mathfrak{g}_0$ denotes the reductive part of $P_{\alpha}^{-}$
and has the Dynkin diagram $\Gamma\setminus \{\alpha\}$. Then $i$ induces an inclusion 
\[\mathfrak{g}_{0}\hookrightarrow L_\alpha.\]
We will also denote this embedding by $i$. 

\begin{cor}\label{special Manin} 
Let $S:=\Gamma\setminus \{\alpha\}$
and $\zeta_S:=\{h\in\mathfrak{h}:\beta(h)=0, \forall \beta\in S\}$. 
For any Manin triple $(Q',\Delta_{\alpha},\mathfrak{l})$, there exists a unique
generalized Belavin-Drinfeld data
$\mathcal{B}\mathcal{D}=(A,A',\mathfrak{i}_{\mathfrak{a}},\mathfrak{i}_{\mathfrak{a}'})$
where $A:i(S)\times \{0\}\longrightarrow \{0\}\times S$,
$A(i(\gamma),0)=(0,\gamma)$ and 
$\mathfrak{i}_{\mathfrak{a}}=\rm{diag}(\zeta_S)$,
such that $(Q',\Delta_{\alpha},\mathfrak{l})$ is conjugate to the Manin triple
$\mathcal{T}_{\mathcal{B}\mathcal{D}}=(Q',\mathfrak{i},\mathfrak{i'})$. Moreover, up to a conjugation which preserves $\Delta_{\alpha}$, one has 
$\mathfrak{l}=\mathfrak{i'}$. 
\end{cor}
\begin{proof}
Let us suppose that $(Q',\Delta_{\alpha},\mathfrak{l})$ is a Manin triple. Then there exists a unique generalized Belavin-Drinfeld data 
$\mathcal{B}\mathcal{D}=(A,A',\mathfrak{i}_{\mathfrak{a}},\mathfrak{i}_{\mathfrak{a}'})$ such that the corresponding $\mathcal{T}_{\mathcal{B}\mathcal{D}}=(Q',\mathfrak{i},\mathfrak{i'})$ is conjugate to $(Q',\Delta_{\alpha},\mathfrak{l})$. Since $\mathfrak{i}$ and $\Delta_{\alpha}$ are conjugate and $\Delta_{\alpha}$ is ``under'' the parabolic subalgebra $P_{\alpha_{0}}^{+}\times P_{\alpha}^{-}$, it follows that 
$\mathfrak{i}$ is also ``under'' this parabolic and thus $\mathfrak{a}=
\zeta_S\times\zeta_S$. According to \cite{Del}, p. 136, the map $A$ should be an isometry between $i(S)\times\{0\}$ and $\{0\}\times S$. Let us write
$A(i(\gamma),0)=(0,\tilde{A}(\gamma))$, where $\tilde{A}:S\rightarrow S$ is an
isometry which will be determined below.

Let $\mathfrak{m}$ be the image of $\mathfrak{g}_0$ in $L_\alpha$ via the embedding $i$. Then $\mathfrak{m}$ is spanned by $X_{i(\beta)}$, $H_{i(\beta)}$, $Y_{i(\beta)}$ for all $\beta\in S$. 

According to Theorem \ref{Manin}, the Lagrangian subspace $\mathfrak{i}$ contains $\mathfrak{k}:=\{(X,\tau(X)):X\in\mathfrak{m}\}$, where $\tau$ satisfies the following conditions: 
$\tau(X_{i(\beta)})=X_{\tilde{A}\beta}$, $\tau(H_{i(\beta)})=H_{\tilde{A}\beta}$, $\tau(Y_{i(\beta)})=Y_{\tilde{A}\beta}$, for all $\beta\in S$. 

We obtain
$\tau i(X_{\beta})=X_{\tilde{A}\beta}$,
$\tau i(Y_{\beta})=Y_{\tilde{A}\beta}$,
$\tau i(H_{\beta})=H_{\tilde{A}\beta}$, for all $\beta\in S$.

Since $\mathfrak{i}$ and $\Delta_{\alpha}$ are conjugate, we must have 
$\mathfrak{i}_{\mathfrak{a}}=\rm{diag}(\zeta_S)$. Moreover
$\tau i$ has to be an inner automorphism of $\mathfrak{g}_0$. 
It follows that the isometry $\tilde{A}:S\rightarrow S$ which corresponds to this inner automorphism must be the identity. 
Thus $\tilde{A}=\rm{id}$ and this ends the proof.
\end{proof}

We will consider triples of the form $(\Gamma'_{1},\Gamma'_{2},\tilde{A'})$, where $\Gamma'_{1}\subseteq\Gamma^{ext} \setminus \{\alpha\}$, 
$\Gamma'_{2}\subseteq\Gamma$ and $\tilde{A'}$ is an isometry between $\Gamma'_{1}$ and $\Gamma'_{2}$.
\begin{defn}\label{types} 
We say that a triple $(\Gamma'_{1},\Gamma'_{2},\tilde{A'})$
is of \textit{type I} if $\alpha\notin\Gamma'_{2}$ and $(\Gamma'_{1},i(\Gamma'_{2}),i\tilde{A'})$ is an admissible triple in the sense of \cite{BD}.
The triple $(\Gamma'_{1},\Gamma'_{2},\tilde{A'})$ is of \textit{type II} if $\alpha\in\Gamma'_{2}$ and 
$\tilde{A'}(\beta)=\alpha$, for some
$\beta\in\Gamma'_{1}$ and $(\Gamma'_{1}\setminus\{\beta\},
i(\Gamma'_{2}\setminus\{\alpha\}),i\tilde{A'})$ is an
admissible triple in the sense of \cite{BD}.

\end{defn}

Using the definition of generalized Belavin-Drinfeld data, one can easily check the following:
\begin{lem}
Let  $A:i(S)\times \{0\}\longrightarrow \{0\}\times S$,
$A(i(\gamma),0)=(0,\gamma)$ and $\mathfrak{i}_{\mathfrak{a}}=\rm{diag}(\zeta_S)$.
A quadruple $(A,A', \mathfrak{i}_{\mathfrak{a}},\mathfrak{i}_{\mathfrak{a'}})$ is
generalized Belavin-Drinfeld data if and only if the pair
$(A',\mathfrak{i}_{\mathfrak{a'}})$ satisfies the following conditions:

(1) $A':\Gamma'_{1}\times \{0\} \longrightarrow \{0\}\times \Gamma'_{2}$ is given by 
$A'(\gamma,0)=(0,\tilde{A'}(\gamma))$ and $(\Gamma'_{1},\Gamma'_{2},\tilde{A'})$ is of type I or II from above.

(2) Let $\mathfrak{f}$ be the subspace of $\mathfrak{h}\times\mathfrak{h}$
spanned by pairs $(H_{i(\gamma)},H_{\gamma})$ for all $\gamma\in S$ and 
$\mathfrak{f'}$ be the subspace of $\mathfrak{h}\times\mathfrak{h}$ spanned by pairs
$(H_{\beta},H_{\tilde{A'}(\beta)})$ for all $\beta\in\Gamma'_{1}$. Let
$\mathfrak{i}_{\mathfrak{a'}}$ be Lagrangian subspace of
$\mathfrak{a'}:=\{(h_{1},h_{2})\in\mathfrak{h}\times\mathfrak{h}:\beta(h_{1})=0,
\gamma(h_{2})=0, \forall\beta\in\Gamma'_{1}, \forall \gamma\in\Gamma'_{2}\}$. Then
\begin{equation}\label{condf}
(\mathfrak{f'}\oplus\mathfrak{i}_{\mathfrak{a'}})\cap(\mathfrak{f}\oplus\mathfrak{i}_{\mathfrak{a}})={0}.
\end{equation}
\end{lem}
\begin{rem}
One can always find $\mathfrak{i}_{\mathfrak{a'}}$ which is a Lagrangian subspace of $\mathfrak{a'}$ and satisfies condition (\ref{condf}). This is a consequence of \cite{Del} Remark 2, p. 142. A proof of this elementary fact can also be found in \cite{S3}, Lemma 5.2.
\end{rem}
Summing up the previous results we conclude the following:

\begin{thm}\label{main}
Let $\alpha$ be a simple root. Suppose that $\mathfrak{l}$ is a Lagrangian subalgebra of $L_{\alpha}\times\mathfrak{g}$ transversal to $\Delta_{\alpha}$. Then, up to a conjugation which preserves $\Delta_{\alpha}$, one has $\mathfrak{l}=\mathfrak{i'}$, where $\mathfrak{i'}$ is constructed from a pair formed by 
a triple $(\Gamma'_{1},\Gamma'_{2},\tilde{A'})$ of type I or II and
a Lagrangian subspace $\mathfrak{i}_{\mathfrak{a'}}$ of
$\mathfrak{a'}$ such that (\ref{condf}) is satisfied. 

\end{thm}
\begin{rem}
The above algorithm covers also the case when the simple root $\alpha$ 
has coefficient $k=1$ in the decomposition of $\alpha_{\rm{max}}$,
which was considered in \cite{KPSST}. In this form we do not need to
use the Cartan involution which appeared there.
\end{rem}
\begin{rem}
In \cite{KPSST} we constructed two examples of quasi-trigonometric solutions for $\mathfrak{g}=\mathfrak{sl}(n+1)$ related to the Cremmer-Gervais $r$-matrix. 
Let us apply the present algorithm for the simple root $\alpha_{n}$. 
Let $\Gamma=\{\alpha_{1},...,\alpha_{n}\}$ and $\alpha_{0}=-\alpha_{\rm{max}}$.

We consider the triple $(\Gamma'_{1},\Gamma'_{2},\tilde{A'})$, where
$\Gamma'_{1}=\{\alpha_{0},\alpha_{1},...\alpha_{n-2}\}$,
$\Gamma'_{2}=\{\alpha_{1},...\alpha_{n-1}\}$ and $\tilde{A'}(\alpha_j)=\alpha_{j+1}$, for $j=0,...,n-2$.
This is a triple of type I, cf. Definition \ref{types}. 

Another example is given by the following triple:
$\Gamma'_{1}=\{\alpha_{0},\alpha_{1},...\alpha_{n-1}\}$,
$\Gamma'_{2}=\{\alpha_{1},...\alpha_{n}\}$ and $\tilde{A'}(\alpha_{j})=\alpha_{j+1}$,
for $j=1,...,n-1$. This is a triple of type II, cf. Definition \ref{types}. 

The above triples together with the corresponding subspaces $\mathfrak{i}_{\mathfrak{a}'}$ (which we do not compute here) induce two Lagrangian subalgebras of Cremmer-Gervais type. 

\end{rem}

Theorem \ref{main} can also be applied to roots which have coefficient
greater than one. 
As an example, let us classify solutions in $\mathfrak{g}=o(5)$.
\begin{cor}
Let $\alpha_1$, $\alpha_2$ be the simple roots in $o(5)$ and
$\alpha_0=-2\alpha_{1}-\alpha_{2}$. Up to gauge equivalence, there exist
two quasi-trigonometric solutions with non-trivial polynomial part.
\end{cor}
\begin{proof}
The root $\alpha_{2}$ has coefficient $k=1$ in the decomposition of the maximal root. The only possible choice for a triple 
$(\Gamma'_{1},\Gamma'_{2},\tilde{A'})$ with 
$\Gamma'_{1}\subseteq\{\alpha_{0},\alpha_{1}\}$, 
$\Gamma'_{2}\subseteq\{\alpha_{1},\alpha_{2}\}$ to be of type I or II is 
$\Gamma'_{1}=\{\alpha_{0},\}$, $\Gamma'_{2}=\{\alpha_{2},\}$, 
$\tilde{A'}(\alpha_{0})=\alpha_{2}$. One can check that
$\mathfrak{i}_{\mathfrak{a'}}$ is a 1-dimensional space spanned by
the following pair 
$(\rm{diag}(-2,1,0,-1,2),\rm{diag}(0,\sqrt{5},0,-\sqrt{5},0))$. 
 The Lagrangian subalgebra $\mathfrak{i}'_{2}$ constructed from this triple is transversal to $\Delta_{\alpha_{2}}$ in $\mathfrak{g}\times\mathfrak{g}$. 

The root $\alpha_{1}$ has coefficient $k=2$ in the decomposition of the maximal root. The only possible choice for a triple 
$(\Gamma'_{1},\Gamma'_{2},\tilde{A'})$ with 
$\Gamma'_{1}\subseteq\{\alpha_{0},\alpha_{2}\}$, 
$\Gamma'_{2}\subseteq\{\alpha_{1},\alpha_{2}\}$ is again
$\Gamma'_{1}=\{\alpha_{0},\}$, $\Gamma'_{2}=\{\alpha_{2},\}$, 
$\tilde{A'}(\alpha_{0})=\alpha_{2}$ and 
$\mathfrak{i}_{\mathfrak{a'}}$ is as in the previous case. 
The Lagrangian subalgebra $\mathfrak{i}'_{1}$ constructed from this triple is transversal to $\Delta_{\alpha_{1}}$ in 
$L_{\alpha_{1}}\times\mathfrak{g}$.

\end{proof}

\section{Quantization of quasi-trigonometric solutions}
Let us consider the Lie bialgebra structure on $\mathfrak{g}[u]$ given by 
the simplest quasi-trigonometric solution: 
\begin{equation}\label{Delta}
\delta(p(u))=[r(u,v),p(u)\otimes 1+1\otimes p(v)],
\end{equation}
with $r(u,v)=v\Omega/(u-v) +\Sigma_{\alpha} e_{\alpha}\otimes f_{\alpha}+{\frac{1}{2}}\Omega_0$,
where ${e_{\alpha},  f_{\alpha}}$ are root vectors of $\mathfrak{g}$ and $\Omega_0$ is the Cartan part of  $\Omega$. Here $\delta$ and $r$ are exactly the same as $\delta_3$ and $r_3$ from the Introduction.

Quasi-trigonometric solutions correspond to Lie bialgebra
structures on $\mathfrak{g}[u]$ which are obtained by twisting $\delta$.
The quantization of the Lie bialgebra 
$(\mathfrak{g}[u],\delta)$ is the quantum algebra $U_q(\mathfrak{g}[u])$ 
introduced by V. Tolstoy in \cite{tol}. We will recall its construction with a
slight modification. 

Let $\mathfrak{g}$ be a finite-dimensional complex simple Lie algebra
of rank $l$ with a standard Cartan matrix $A=(a_{ij})_{i,j=1}^l$, with
a system of simple roots $\Gamma= \{\alpha_1,\ldots, \alpha_l\}$, and with a
Chevalley basis $h_{\alpha_i}$, $e_{\pm\alpha_i}$ $(i=1,2,\ldots, l)$.
Let $\theta$ be the maximal (positive) root of $\mathfrak{g}$. The corresponding 
non-twisted affine algebra with zero central charge $\hat{\mathfrak{g}}$ is generated by
$\mathfrak{g}$ and the additional affine elements $e_{\delta-\theta}:=ue_{-\theta}$,
$e_{-\delta+\theta}:=u^{-1}e_{\theta}$.

The Lie algebra $\mathfrak{g}[u]$ is generated by $\mathfrak{g}$,
the positive root vector $e_{\delta-\theta}$ and the Cartan element
$h_{\delta-\theta}=[e_{\delta-\theta},e_{-\delta+\theta}]$. The standard
defining relations of the universal
enveloping algebra $U(\mathfrak{g}[u])$ are given by the formulas:
\begin{eqnarray}
[h_{\alpha_{i}}^{},h_{\alpha_{j}}^{}]\!\!&=\!\!&0~,
\label{tolsA1'}
\\[5pt]
[h_{\alpha_{i}}^{},e_{\pm\alpha_{j}}]\!\!&=\!\!&
\pm(\alpha_{i},\alpha_{j})e_{\pm\alpha_{j}}~,
\label{tolsA2}
\\[5pt]
[e_{\alpha_{i}},e_{-\alpha_{j}}]\!\!&=\!\!&\delta_{ij}h_{\alpha_{i}}~,
\label{tolsA2'}
\\[5pt]
(\mathrm{ad} e_{\pm\alpha_{i}}\!)^{n_{ij}} e_{\pm\alpha_{j}}\!\!&=\!\!&0
\qquad{\rm for}\;\;i\neq j,\;\;n_{ij}\!:=\!1\!-\!a_{ij}~,
\label{tolsA3}
\\[5pt]
[h_{\alpha_i}^{},e_{\delta-\theta}]\!\!&=\!\!&-
(\alpha_i,\theta)\,e_{\delta-\theta}'~,
\label{tolsA3'}
\\[5pt]
[e_{-\alpha_i},e_{\delta-\theta}]\!\!&=\!\!&0~,
\label{tolsA4}
\\[5pt]
(\mathrm{ad} e_{\alpha_i})^{n_{i0}}e_{\delta-\theta}\!\!&=\!\!&0
\quad\;{\rm for}\;\,n_{i0}=1\!+\!2(\alpha_i,\theta)/(\alpha_i,\alpha_i),
\label{tolsA5}
\\[5pt]
[[e_{\alpha_i},e_{\delta-\theta}],e_{\delta-\theta}]\!\!&=\!\!&0\quad\;
{\rm for}\;\,\mathfrak{g}\ne\mathfrak{sl}_2 \;\;{\rm and}\;\;(\alpha_i,\theta)\ne 0~,
\label{tolsA6}
\\[5pt]
[[[e_{\alpha},e_{\delta-\alpha}],e_{\delta-\alpha}],e_{\delta-\alpha}]
\!\!&=\!\!&0\quad\;{\rm for}\;\,\mathfrak{g}=\mathfrak{sl}_2
\;\;(\theta=\alpha)~.
\label{tolsA7}
\end{eqnarray}

The quantum algebra $U_q(\mathfrak{g}[u])$ is 
a $q$-defor\-mation of $U(\mathfrak{g}[u])$.
The Chevalley generators for  $U_q(\mathfrak{g}[u])$ are 
$k_{\alpha_i}^{\pm 1}:=q^{\pm h_{\alpha_i}}$, $e_{\pm\alpha_i}$
$(i=1,2,\ldots, l)$, $e_{\delta-\theta}$
and $k_{\delta-\theta}:=q^{h_{\delta-\theta}}$.  Then the defining relations of
$U_q(\mathfrak{g}[u])$ are the following:
\begin{eqnarray}
k_{\alpha_i}^{\pm 1}k_{\alpha_j}^{\pm 1}\!\!&=\!\!&
k_{\alpha_j}^{\pm 1}k_{\alpha_i}^{\pm 1}~,
\label{tolsA8'}
\\[5pt]
k_{\alpha_i}^{}k^{-1}_{\alpha_i}\!\!&=\!\!&
k^{-1}_{\alpha_i}k_{\alpha_i}^{}=1~,
\label{tolsA9}
\\[5pt]
k_{\alpha_i}^{}e_{\pm\alpha_j}^{}k^{-1}_{\alpha_i}\!\!&=\!\!&
q^{\pm(\alpha_i,\alpha_j)}e_{\pm\alpha_j}^{}~,
\label{tolsA9'}
\\[5pt]
[e_{\alpha_i}^{},e_{-\alpha_i}^{}]\!\!&=\!\!&
\frac{k_{\alpha_i}-k_{\alpha_i}^{-1}}{q-q^{-1}}~,
\label{tolsA10}
\\[5pt]
(\mathrm{ad}_{q}e_{\pm\alpha_{i}}^{}\!)^{n_{ij}}e_{\pm\alpha_{j}^{}}\!\!&=\!\!&0
\quad{\rm for}\;\,i\neq j,\;\;n_{ij}\!:=\!1\!-\!a_{ij}~,
\label{tolsA11}
\\[5pt]
k_{\alpha_i}^{}e_{\delta-\theta}^{}k^{-1}_{\alpha_i}\!\!&=\!\!&
q^{-(\alpha_i,\theta)}e_{\delta-\theta}^{}~,
\label{tolsA11'}
\\[5pt]
[e_{-\alpha_i}^{},e_{\delta-\theta}^{}]\!\!&=\!\!&0~,
\label{tolsA12}
\\[5pt]
(\mathrm{ad}_{q}e_{\alpha_i}^{})^{n_{i0}}e_{\delta-\theta}^{}\!\!&=\!\!&0
\quad{\rm for}\;\;n_{i0}=1\!+\!2(\alpha_i,\theta)/(\alpha_i,\alpha_i),
\label{tolsA13}
\\[5pt]
[[e_{\alpha_i}^{},e_{\delta-\theta}^{}]_{q},
e_{\delta-\theta}^{}]_{q}\!\!&=\!\!&0\quad{\rm for}\;\;
\mathfrak{g}\ne\mathfrak{sl}_2\;\;{\rm and}\;\;(\alpha_i,\theta)\ne0~,
\label{tolsA14}
\\[5pt]
[[[e_{\alpha}^{},e_{\delta-\alpha}^{}]_{q},
e_{\delta-\alpha}^{}]_{q},e_{\delta-\alpha}^{}]_{q}\!\!&=\!\!&0
\quad{\rm for}\;\;\mathfrak{g}=\mathfrak{sl}_2~,
\label{tolsA15}
\end{eqnarray}
where $(\mathrm{ad}_q e_{\beta})e_{\gamma}$ is the q-commutator:
\begin{equation}
\quad\quad\;{}
(\mathrm{ad}_q e_{\beta}^{})e_{\gamma}^{}:= [e_\beta^{},e_\gamma^{}]_q:=
e_\beta^{}e_\gamma^{}-q^{(\beta,\gamma)}e_\gamma^{}e_\beta^{}~.
\label{tolsA16}
\end{equation}
The comultiplication $\Delta_{q}$, the antipode $S_{q}$, and
the co-unit $\varepsilon_{q}$ of $U_{q}(\mathfrak{g}[u])$
are given by
\begin{equation}
\begin{array}{rcl}
\Delta_{q}(k_{\alpha_i}^{\pm 1})\!\!&=\!\!&
k_{\alpha_i}^{\pm 1}\otimes k_{\alpha_i}^{\pm 1}~
\\[5pt]
\Delta_{q}(e_{-\alpha_i}^{})\!\!&=\!\!&
e_{-\alpha_i}^{}\otimes k_{\alpha_i}^{}+1\otimes
e_{-\alpha_i}^{}~,
\\[5pt]
\Delta_{q}(e_{\alpha_i}^{})\!\!&=\!\!&
e_{\alpha_i}^{}\otimes1+k_{\alpha_i}^{-1}\otimes
e_{\alpha_i}^{}~,\qquad
\\[5pt]
\Delta_{q}(e_{\delta-\theta}^{})\!\!&=\!\!&e_{\delta-\theta}^{}
\otimes1+k_{\delta-\theta}^{-1}\otimes
e_{\delta-\theta}^{}~,\qquad
\label{tolsA19}
\end{array}
\end{equation}
\begin{equation}
\begin{array}{rcl}
S_{q}(k_{\alpha_i}^{\pm 1})\!\!&=\!\!&k_{\alpha_i}^{\mp 1}~,
\\[5pt]
S_{q}(e_{-\alpha_i}^{})\!\!&=\!\!&-e_{-\alpha_i}^{}k_{\alpha_i}^{-1}~,
\\[5pt]
S_{q}(e_{\alpha_i}^{})\!\!&=\!\!&-k_{\alpha_i}^{}e_{\alpha_i}^{}~,
\\[5pt]
S_{q}(e_{\delta-\theta}^{})\!\!&=\!\!&
-k_{\delta-\theta}^{}e_{\delta-\theta}^{}~,
\end{array}
\end{equation}
\begin{equation}
\varepsilon_{q}(e_{\pm\alpha_i})~=~\varepsilon_{q}(e_{\delta-\theta})=0~,
\qquad\varepsilon_{q}(k_{\alpha_i}^{\pm 1})=1.
\label{tolsA22'}
\end{equation}
We see that
\begin{equation}
\qquad\;{}k_{\delta-\theta}^{}=\,k_{\alpha_1}^{-n_1}
k_{\alpha_2}^{-n_2}\cdots k_{\alpha_l}^{-n_l}
\label{tolsA23}
\end{equation}
if $\theta=n_1^{}\alpha_1^{}+n_2^{}\alpha_2^{}+\cdots\,+ n_l^{}\alpha_l^{}$.
\par

Let us consider now $U_{\hbar}(\mathfrak{g}[u])$, which is an algebra over 
$\mathbb{C}[\hbar]]$ defined by the previous relations in which we set 
$q=\mathrm{exp}(\hbar)$. 

\begin{thm}
The classical limit of $U_{\hbar}(\mathfrak{g}[u])$ is 
$(\mathfrak{g}[u],\delta)$.
\end{thm}
\begin{proof}
Since the image of $e_{\delta-\theta}$ in $U_{\hbar}(\mathfrak{g}[u])/{\hbar}U_{\hbar}(\mathfrak{g}[u])$ is $ue_{\theta}$, we can identify the classical
limit of $U_{\hbar}(\mathfrak{g}[u])$ with $U(\mathfrak{g}[u])$.

It remains to prove that for any $a\in\mathfrak{g}[u]$ and its preimage 
$\tilde{a}\in U_{\hbar}(\mathfrak{g}[u])$, we have 
\[{\hbar}^{-1}(\Delta(\tilde{a})-\Delta^{op}(\tilde{a}))\;\; \mathrm{mod}\;\; \hbar=
\delta(a).\]

It is clear that $U_{\hbar}(\mathfrak{g}[u])$ is a Hopf subalgebra 
of $U_{\hbar}(\hat{\mathfrak{g}})$ with zero central charge.
It follows from \cite{D} that 
$\Delta^{op}(\tilde{a})=R\Delta(\tilde{a})R^{-1}$,
where $R=1\otimes 1+\hbar r+...$ is the universal $R$-matrix for  $U_{\hbar}(\hat{\mathfrak{g}})$ (see also \cite{KT} and \cite{J}). 

Writing $\tilde{a}=a+\hbar s$, we see that 
\[{\hbar}^{-1}(\Delta(\tilde{a})-R\Delta(a)R^{-1})\;\; \mathrm{mod}\;\; \hbar=
[r,a\otimes 1+1\otimes a]=\delta(a).\]
This ends the proof of the theorem.
\end{proof}

The Lie bialgebra structures induced by quasi-trigonometric solutions
are obtained by twisting $\delta$ via a classical twist. 
According to Theorem \ref{main}, any such
twist is given by a triple $(\Gamma'_{1},\Gamma'_{2},\tilde{A'})$ of type I or II together with a Lagrangian subspace $\mathfrak{i}_{\mathfrak{a'}}$ of
$\mathfrak{a'}$ such that (\ref{condf}) is satisfied.

On the other hand, Theorem \ref{Hal} states that classical twists can be extended to quantum twists. 

We conclude that any data of the form $(\Gamma'_{1},\Gamma'_{2},\tilde{A'},\mathfrak{i}_{\mathfrak{a'}})$ provides a twisted comultiplication and antipode in the quantum algebra $U_{\hbar}(\mathfrak{g}[u])$.
In case $\mathfrak{g}=\mathfrak{sl}(n)$, some exact formulas were obtained in 
\cite{KPSST}. 

\vspace*{1cm}

\textbf{Acknowledgement.} The authors are thankful to V. Tolstoy 
for valuable discussions. 

\bibliographystyle{amsalpha}

\end{document}